\documentclass[11pt,reqno]{amsart}
\usepackage{amsmath,amsthm,amssymb,mathrsfs,stmaryrd,color}
\usepackage[all]{xy}
\usepackage{url}
\usepackage{slashed}
\usepackage{geometry}

\usepackage[utf8]{inputenc}
\usepackage[T1]{fontenc}

\usepackage{relsize}
\usepackage[bbgreekl]{mathbbol}
\usepackage{amsfonts}
\usepackage{enumitem}
\setlist[enumerate]{itemsep=2pt,parsep=2pt,before={\parskip=2pt}}
\usepackage[colorlinks=true,hyperindex,linkcolor=magenta,pagebackref=false,citecolor=cyan,pdfpagelabels]{hyperref}

\DeclareMathOperator{\RGamma}{R\Gamma}
\DeclareMathOperator{\WCart}{WCart}
\DeclareMathOperator{\Spf}{Spf}

\DeclareMathOperator{\conj}{conj}
\DeclareMathOperator{\Fil}{Fil}
\DeclareMathOperator{\gr}{gr}

\DeclareMathOperator{\Frac}{Frac}
\DeclareMathOperator{\trdeg}{trdeg}
\DeclareMathOperator{\cd}{cd}

\DeclareMathOperator{\im}{im}
\newcommand{\Z}{\mathbf Z}
\newcommand{\F}{\mathbf F}

\newcommand{\calO}{\mathcal O}

\DeclareSymbolFontAlphabet{\mathbb}{AMSb}
\DeclareSymbolFontAlphabet{\mathbbl}{bbold}
\newcommand{\Prism}{\mathbbl{\Delta}}

\newcommand{\Fp}{\mathbf{F}_p}
\newcommand{\Zp}{\mathbf{Z}_p}
\newcommand{\Prismbar}{\overline{\Prism}}
\newcommand{\wcartHT}{\WCart^{\mathrm{HT}}}
\newcommand{\Dle}[1]{\mathcal{D}^{\leq #1}}

\usepackage{cleveref}
\newtheorem{theorem}{Theorem}[section]
\newtheorem{proposition}[theorem]{Proposition}
\newtheorem{lemma}[theorem]{Lemma}

\newtheorem{conjecture}[theorem]{Conjecture}
\theoremstyle{remark}
\newtheorem{remark}[theorem]{Remark}
\crefname{remark}{Remark}{Remarks}
\crefname{theorem}{Theorem}{Theorems}
\crefname{proposition}{Proposition}{Propositions}
\crefname{lemma}{Lemma}{Lemmas}
\crefname{corollary}{Corollary}{Corollaries}
\crefname{conjecture}{Conjecture}{Conjectures}
\crefname{section}{Section}{Sections}

\begin{document}

\title{A Counterexample to Bhatt--Lurie's Cohomological Dimension Conjecture}
\author{Guo Li}
\address{School of Mathematical Sciences, University of Chinese Academy of Sciences, No. 19A Yuquan Road,  Beijing 100049, China}
\email{liguo22@mails.ucas.ac.cn}
\date{}

\begin{abstract}
We exhibit a counterexample to a conjecture of Bhatt--Lurie on the cohomological dimension of the Hodge--Tate locus for regular local rings. The example arises from a non-excellent discrete valuation ring constructed by Datta--Smith, closely related to an earlier example of Bosch--L\"{u}tkebohmert--Raynaud. We also explain how the same mechanism yields broader families of counterexamples, while the expected bound is recovered under an excellence hypothesis.
\end{abstract}

\maketitle

\tableofcontents

\section{Introduction}

A fundamental contribution of \cite{BL22a,BL22b,drinfeld2024prismatization} is the geometric realization of prismatic cohomology, originally introduced by Bhatt and Scholze \cite{BS22} as a unifying framework for $p$-adic cohomology theories. More precisely, to every $p$-adic formal scheme $X$ one associates the Cartier--Witt stack  $\WCart_X$, its \emph{prismatization}. Under suitable hypotheses, quasicoherent sheaves on $\WCart_X$ can be identified with prismatic crystals on $X$. This viewpoint yields, among other things, a conceptual reproof of the Hodge--Tate comparison theorem, a stack-theoretic interpretation of the prismatic logarithm, and a geometric framework for studying Hodge--Tate phenomena of $p$-adic formal schemes via the Hodge--Tate locus
\[
\wcartHT_X \subset \WCart_X.
\]

Motivated by the heuristic that regular rings behave as if they were formally smooth ``over $\F_1$'', Bhatt and Lurie formulate the following conjecture.

\begin{conjecture}[{\cite[Conjecture~10.1]{BL22b}}]\label[conjecture]{conj:original}
Let $R$ be a $p$-complete noetherian regular local ring with perfect residue field. Then the functor $\RGamma(\wcartHT_{\Spf(R)}, -)$ carries $\Dle{0}$ to $\Dle{\dim(R)}$.
\end{conjecture}

In this note, we show that the non-excellent DVR $V_p$ of equal characteristic $p$ constructed in \cite{datta2018excellenceprimecharacteristic} (see also \cref{sec:Vp}) is a counterexample to \cref{conj:original}. The basic idea of this construction goes back to an earlier example in Bosch--L\"{u}tkebohmert--Raynaud, see \cref{rem:BLR}.

\begin{theorem}[Counterexample]\label[theorem]{thm:main}
We have $H^2(\wcartHT_{\Spf(V_p)}, \calO) \simeq \Omega^2_{V_p/\Fp}\{-2\} \neq 0$. In particular, \cref{conj:original} is false.
\end{theorem}

The computation is carried out in \cref{sec:comp}. We also explain why the limit argument proposed in \cite[Remark~10.2]{BL22b} fails for $V_p$ (see \cref{rem:limit-failure}). Moreover,
\begin{itemize}
    \item The counterexample is quite flexible: the construction produces uncountably many DVRs, works over function fields of arbitrarily large transcendence degree, and extends to regular local rings of arbitrary Krull dimension (see \cref{rem:more-examples}).
    \item On the positive side, we observe that the conjecture holds under the additional hypothesis that $R$ is excellent, a result essentially contained in \cite{BM_2023} (see \cref{rem:excellent-repair}). This gives a natural repair of  \cref{conj:original}.
\end{itemize}

\noindent\textbf{Acknowledgments.}
The author thanks Yichao Tian, Jiahong Yu, and Zhouhang Mao for valuable comments on this note. In particular, Jiahong Yu pointed out that the result in \cref{prop:Hn} can be derived from the general theory of $p$-Cartier smooth algebras developed in \cite{B23} (see \cref{rem:relation-B23} for a discussion). 
The author also thanks the Rethlas system developed at Peking University for assistance in locating relevant references, including the notion of Cartier smoothness in \cite{BL22a} and the alternative proof of \cref{lem:regular} in \cite{popescu2021valuationringsdimensionlimits}.

\section{The DVR $V_p$}\label[section]{sec:Vp}

Virtually all Noetherian rings arising in algebraic geometry and number theory are excellent.
We refer to \cite[Section~34]{Matsumura1980CommutativeAlgebra} for the definition and basic properties of excellent rings. An important fact is that any finitely generated algebra (hence taking quotient) over an excellent ring is again excellent.

We recall the construction of Datta--Smith \cite[Section~4.1]{datta2018excellenceprimecharacteristic}. Let $k = \overline{\Fp}$ and $K = k(x,y)$. Choose a non-unit power series $p(t) \in k[[t]]$ which is not algebraic over $k(t)$; such power series exist because $k(t)$ is countable whereas $k[[t]]$ is uncountable. The ring homomorphism
\[
k[x,y] \hookrightarrow k[[t]], \qquad x \mapsto t,\; y \mapsto p(t)
\]
is injective (otherwise $p(t)$ would be algebraic over $k(t)$) and induces an inclusion $K \hookrightarrow k((t))$ of fraction fields. Restricting the $t$-adic valuation on $k((t))$ to $K$ yields a discrete valuation with value group $\Z$; its valuation ring is
\[
V_p := k[[t]] \cap K,
\]
with maximal ideal generated by $x$ (since $v(x) = 1$). The inclusion $V_p \hookrightarrow k[[t]]$ induces $V_p/\mathfrak{m}_{V_p} \cong k$.

\begin{lemma}\label[lemma]{lem:Vp}
The ring $V_p$ satisfies:
\begin{enumerate}
\item $V_p$ is a DVR of equal characteristic $p$; in particular, a noetherian regular local ring of Krull dimension $1$ with perfect residue field $k$.
\item $\Frac(V_p) = k(x,y)$ and $V_p/\mathfrak{m}_{V_p} \cong k$.
\item $V_p$ is not excellent and not $F$-finite.
\item $V_p$ is $p$-complete but not complete with respect to its maximal ideal.
\item $V_p$ is not smooth over $k$.
\end{enumerate}
\end{lemma}

\begin{proof}
Items (1)--(3) are proved in \cite[Section~4.1]{datta2018excellenceprimecharacteristic}. For (4): $p = 0$ in $V_p$ gives $p$-completeness trivially. By \cite[\href{https://stacks.math.columbia.edu/tag/07QW}{Tag 07QW}]{stacks-project}, Noetherian complete local rings are excellent, hence $V_p$ is not complete with respect to its maximal ideal by (3). For (5): $k = \overline{\Fp}$ is perfect, hence excellent; any smooth algebra (in particular of finite type) over an excellent ring is excellent, while $V_p$ is not excellent.
\end{proof}

\begin{remark}
By the above lemma, the description of $\wcartHT_{X}$ as a classifying stack $BG$ in \cite[Proposition~9.5]{BL22b} does not apply to $V_p$. We therefore adopt a different approach to the computation.
\end{remark}

\begin{remark}\label[remark]{rem:BLR}
A predecessor of the above construction appears in
\cite[Example~11 of Section~3.6]{BLR90}. In that example, one chooses an element $\xi\in \mathbb F_p[[T]]$ transcendental over
$\mathbb F_p(T)$. Consider
\[
R_{\mathrm{BLR}} := \mathbb F_p(T,\xi^p)\cap \mathbb F_p[[T]],
\]
where the intersection is taken inside $\mathbb F_p((T))$. This is shown there to be a DVR which is not excellent.
Thus the basic philosophy of producing non-excellent DVRs by intersecting a function field inside a formal power series ring goes back at least to \cite{BLR90}. We use the Datta--Smith formulation in this note  because it is convenient for producing the variants considered in \cref{rem:more-examples}.
\end{remark}

\begin{remark}
  A famous example of a Noetherian but non-excellent ring was constructed by Nagata: a two-dimensional local Noetherian domain which is catenary but not universally catenary, hence not excellent by definition. This example is recorded in \cite[\href{https://stacks.math.columbia.edu/tag/02JE}{Tag 02JE}]{stacks-project}; see also \cite[Example 14.1]{Matsumura1980CommutativeAlgebra}. However, Nagata's example does not serve our purpose.
\end{remark}

\section{Computation of $H^2(\wcartHT_{\Spf(V_p)}, \calO)$}\label[section]{sec:comp}

\subsection{Cotangent complex and reduction to Hodge--Tate complex}

Recall by \cite[\href{https://stacks.math.columbia.edu/tag/00OD}{Tag 00OD}]{stacks-project}, a Noetherian ring is said to be regular if all its localizations at prime ideals are regular local rings.
\begin{lemma}\label[lemma]{lem:regular fact}
 We have the following key facts:
    \begin{enumerate}
        \item A regular local ring is regular.
        \item Let $\varphi: R\to S$ be a smooth ring map with $R$ regular, then $S$ is regular.
    \end{enumerate}
\end{lemma}
\begin{proof}
   See \cite[\href{https://stacks.math.columbia.edu/tag/0AFS}{Tag 0AFS}]{stacks-project} \& \cite[\href{https://stacks.math.columbia.edu/tag/07NF}{Tag 07NF}]{stacks-project}.
\end{proof}

\begin{lemma}\label[lemma]{lem:cotangent}
$V_p$ is Cartier smooth over $\Fp$ in the sense of \cite[Definition~E.10]{BL22a}. Consequently, $L_{V_p/\Fp}$ is a flat $V_p$-module concentrated in degree $0$.
\end{lemma}

\begin{proof}
By \cref{lem:regular fact}, $V_p$ is a regular ring. One can check directly since a DVR has exactly two prime ideals: both localizations are regular local rings. Hence every regular noetherian $\Fp$-algebra is Cartier smooth over $\Fp$ by \cite[Definition~E.11]{BL22a}. Condition (2) of Cartier smoothness implies that $L\Omega^1_{V_p/\Fp} \simeq L_{V_p/\Fp}$ is flat and concentrated in degree $0$.
\end{proof}

\begin{lemma}\label[lemma]{lem:regular}
$V_p$ is a filtered colimit of smooth $\Fp$-algebras.
\end{lemma}

\begin{proof}
This can be proved in two ways:
\begin{enumerate}
\item The structure map $\Fp \to V_p$ is a regular homomorphism: it is clearly flat. Its unique fiber is $V_p$, whose geometric regularity holds because for any given finite field extension  $L/\Fp$, $V_p\to V_p\otimes_{\Fp}L$ is \'etale and thus $V_p\otimes_{\Fp}L$ is regular by \cref{lem:regular fact}. Thus Popescu's theorem \cite[\href{https://stacks.math.columbia.edu/tag/07GC}{Tag 07GC}]{stacks-project} applies.
\item Alternatively, one can invoke \cite[Theorem~1(1)]{popescu2021valuationringsdimensionlimits} directly.
\end{enumerate}
\end{proof}

\begin{proposition}\label[proposition]{prop:identification}
There is a canonical equivalence
\[
\RGamma\bigl(\wcartHT_{\Spf(V_p)}, \calO\bigr) \;\simeq\; \Prismbar_{V_p/\Zp}.
\]
\end{proposition}

\begin{proof}
We chain the following identifications from \cite{BL22b,BL22a}:
\begin{align*}
\RGamma(\wcartHT_{\Spf(V_p)}, \calO)
&\simeq \RGamma(\wcartHT_{\Spf(V_p)/\Zp}, \calO)
&&\text{(\cite[Remark~5.3]{BL22b})} \\[2pt]
&\simeq \RGamma(\WCart_{\Spf(V_p)/\Zp}, \calO) \otimes_{\Zp}^L \Fp
 \\[2pt]
&\simeq \RGamma_{\Prism}(\Spf(V_p)/\Zp) \otimes_{\Zp}^L \Fp
&&\text{(\cite[Theorem~7.20(2)]{BL22b})} \\[2pt]
&\simeq \Prism_{V_p/\Zp} \otimes_{\Zp}^L \Fp
&&\text{(\cite[Corollary~4.2.6]{BL22a})} \\[2pt]
&=: \Prismbar_{V_p/\Zp}.
\end{align*}
For the third isomorphism: $V_p$ is classical, bounded (as $p=0$), and condition $(\ast)$ holds because  $V_p/pV_p = V_p$ and $L_{V_p/\Fp}$ has Tor-amplitude $[0,0]$ by \cref{lem:cotangent}. Hence \cite[Theorem~7.20(2)]{BL22b} applies. For the fourth isomorphism: the $p$-completion of $V_p$ is itself.
\end{proof}

\subsection{Conjugate filtration and proof of the main theorem}

\begin{proposition}\label[proposition]{prop:Hn}
For every $n \ge 0$,
\[
H^n(\Prismbar_{V_p/\Zp}) \;\simeq\; \Omega^n_{V_p/\Fp}\{-n\}.
\]
\end{proposition}

\begin{proof}
By \cite[Remark~4.1.7]{BL22a}, the Hodge--Tate complex $\Prismbar_{V_p/\Zp}$ carries an exhaustive increasing conjugate filtration $\Fil_{\bullet}^{\conj}$ whose graded pieces satisfy
\begin{equation}\label{eq:HT-gr}
\gr_n^{\conj} \Prismbar_{V_p/\Zp} \;\simeq\; L\widehat{\Omega}^n_{V_p/\Fp}[-n]\{-n\}.
\end{equation}
By \cref{lem:cotangent}, $L_{V_p/\Fp}$ is a flat $V_p$-module in degree $0$. Thus $L\widehat{\Omega}^1_{V_p/\Fp} = L\Omega^1_{V_p/\Fp} \simeq L_{V_p/\Fp}$ is $p$-completely flat. By \cite[Remark~4.1.9]{BL22a}, the conjugate filtration therefore coincides with the Postnikov filtration, whence $\gr_n^{\conj} \simeq H^n[-n]$. Equating with \eqref{eq:HT-gr} yields
\begin{equation}\label{eq:mid}
H^n(\Prismbar_{V_p/\Zp})[-n] \;\simeq\; L\widehat{\Omega}^n_{V_p/\Fp}[-n]\{-n\}.
\end{equation}

It remains to identify $L\widehat{\Omega}^n_{V_p/\Fp}$. By \cref{lem:regular}, write $V_p = \varinjlim_\alpha S_\alpha$ with each $S_\alpha$ a smooth $\Fp$-algebra of finite type. The functor $L\Omega^n_{-/\Fp}$ commutes with sifted (hence filtered) colimits by \cite[Construction~B.1]{BL22a}, so
\[
L\Omega^n_{V_p/\Fp} \simeq \varinjlim_\alpha L\Omega^n_{S_\alpha/\Fp} \simeq \varinjlim_\alpha \Omega^n_{S_\alpha/\Fp}[0]
\]
since each $S_\alpha$ is smooth over $\Fp$. K\"ahler differentials and exterior powers commute with filtered colimits, yielding $\varinjlim_\alpha \Omega^n_{S_\alpha/\Fp}[0] \simeq \Omega^n_{V_p/\Fp}[0]$. Since $p$-completion is trivial in this case, $L\widehat{\Omega}^n_{V_p/\Fp} \simeq \Omega^n_{V_p/\Fp}[0]$. Substituting into \eqref{eq:mid} gives the result. 

Alternatively, Lazard's theorem \cite[\href{https://stacks.math.columbia.edu/tag/058G}{Tag 058G}]{stacks-project} gives the same identification directly from the flatness of $\Omega^1_{V_p/\Fp}$, without invoking Popescu's theorem.
\end{proof}

\begin{lemma}\label[lemma]{lem:omega2} We have
$\Omega^2_{V_p/\Fp} \neq 0$.
\end{lemma}

\begin{proof}
By \cref{lem:Vp}(2), $\Frac(V_p) = k(x,y)$. Since $k/\Fp$ is a filtered colimit of \'etale extensions, $\Omega^1_{k/\Fp} = 0$. Localizing K\"ahler differentials gives
\[
\Omega^1_{V_p/\Fp} \otimes_{V_p} k(x,y)
\simeq \Omega^1_{k(x,y)/\Fp}
\simeq \Omega^1_{k(x,y)/k}
\simeq k(x,y)\,dx \oplus k(x,y)\,dy,
\]
a $k(x,y)$-vector space of dimension $2$. Taking second exterior powers,
\[
\Omega^2_{V_p/\Fp} \otimes_{V_p} k(x,y)
\simeq \bigwedge^2_{k(x,y)} \Omega^1_{k(x,y)/\Fp}
\simeq k(x,y) \neq 0,
\]
so $\Omega^2_{V_p/\Fp} \neq 0$.
\end{proof}

\begin{proof}[\textbf{Proof of \cref{thm:main}}]
\cref{prop:identification} and \cref{prop:Hn} give
\[
H^2(\wcartHT_{\Spf(V_p)}, \calO)
\simeq H^2(\Prismbar_{V_p/\Zp})
\simeq \Omega^2_{V_p/\Fp}\{-2\}.
\]
By \cref{lem:omega2}, this is nonzero. Thus $H^2 \neq 0$, while $\dim(V_p) = 1$, so the cohomological dimension of $\wcartHT_{\Spf(V_p)}$ is at least $2$. By \cref{lem:Vp}, $V_p$ satisfies all the hypotheses of \cref{conj:original}, which is therefore false.
\end{proof}

\begin{remark}[Why the limit argument fails]\label[remark]{rem:limit-failure}
\cite[Remark~10.2]{BL22b} proposes reducing the regular $\Fp$-algebra case to the smooth case via a limit argument. For $V_p$, this reduction cannot yield the bound $\cd \le 1$. Indeed, write $V_p = \varinjlim_j S_j$ as in \cref{lem:regular}, and set $A_j := \im(S_j \to V_p)$. For $j$ sufficiently large, $A_j$ contains both $x$ and $y$, so $\Fp[x,y] \subset A_j$ and both are of finite type over $\Fp$; thus
\[
\dim(A_j) = \dim \Fp[x,y] + \trdeg_{\Fp(x,y)}(\Frac(A_j)) \ge 2.
\]
Hence $\dim(S_j) \ge 2$, and the smooth-case computation does not give the desired bound $\le 1$.
\end{remark}

\begin{remark}[More counterexamples]\label[remark]{rem:more-examples}
The construction generalizes in several directions.
\begin{enumerate}
\item There are uncountably many choices of the transcendental power series $p(t) \in k[[t]]$, each yielding a distinct non-excellent DVR; this follows from the proof of \cite[Corollary~4.4]{datta2018excellenceprimecharacteristic}, which shows that distinct $p(t)$ and $q(t)$ give distinct valuation rings.
\item One may embed $k(x_0, \dots, x_n)$ into $k((t))$ for any $n \ge 2$ by sending $x_0 \mapsto t$ and $x_i \mapsto p_i(t)$ ($i \ge 1$) for power series $p_i(t) \in k[[t]]$ algebraically independent over $k(t)$; such $p_i(t)$ exist because $k(t)$ is countable infinite whereas 
\[
|k((t))|=\max\{|k(t)|,\trdeg(k((t))/k(t))\} \implies \trdeg(k((t))/k(t))=|k((t))|.
\]
This yields a non-excellent DVR $V$ of the function field $k(x_0, \dots, x_n)$ with residue field $k$ and $\dim(V) = 1$. The same proof as that of \cref{thm:main} shows: 
\[
H^{n+1}(\wcartHT_{\Spf(V)}, \calO) \simeq \Omega^{n+1}_{V/\Fp}\{-(n+1)\} \neq 0.
\]

\item The same argument of (2) also admits a countably generated variant, which gives a DVR $V_\infty$ with fraction field $k(x_0,x_1,\ldots)$ and $H^m(\wcartHT_{\Spf(V_\infty)},\calO)\neq 0$ for every finite $m$.

\item Moreover, one can produce counterexamples of \emph{any} Krull dimension $d \ge 1$. Set $$R_d := V_p[t_1, \dots, t_{d-1}]_{(x, t_1, \dots, t_{d-1})}.$$
By \cref{lem:regular fact}, $R_d$ is a regular local ring of dimension $d$ with perfect residue field $k$, and it is trivially $p$-complete.
Since $R_d/(t_1,\dots,t_{d-1})R_d=(V_p)_{(x)}=V_p$ and excellence is stable under taking quotients, $R_d$ cannot be excellent. Its fraction field $\Frac(R_d) = k(x, y, t_1, \dots, t_{d-1})$ has transcendence degree $d+1$ over $k$, so $\Omega^1_{R_d/\Fp}$ has generic rank $d+1$ and $\Omega^{d+1}_{R_d/\Fp} \neq 0$.

The computation of \cref{thm:main} carries over verbatim: $R_d$ is a regular noetherian $\Fp$-algebra by \cref{lem:regular fact}, hence Cartier smooth over $\Fp$, and the conjugate filtration coincides with the Postnikov filtration, yielding
$$H^{d+1}(\wcartHT_{\Spf(R_d)}, \calO) \simeq \Omega^{d+1}_{R_d/\Fp}\{-(d+1)\} \neq 0,$$
while $\dim(R_d) = d$.
\end{enumerate}
\end{remark}

\begin{remark}[Excellent repair]\label[remark]{rem:excellent-repair}
We note that \cref{conj:original} holds under the additional hypothesis that $R$ is excellent. Indeed, set $S = R/pR$. Since excellence is stable under taking quotients, $S$ is a noetherian excellent local $\Fp$-algebra with perfect residue field $k$. \cite[Corollary~2.6]{Kun76} implies that a noetherian local $\Fp$-algebra with perfect residue field is excellent if and only if it is $F$-finite. Hence $S$ is $F$-finite.

Applying \cite[Corollary~4.20]{BM_2023} with $\log_p[k:k^p] = 0$ yields $\cd(\wcartHT_{\Spf(R)}) \le \dim(R)$. Thus the excellent case gives a natural repair of the conjecture, while the example above shows that the original formulation cannot hold without further hypotheses. 

The $F$-finiteness hypothesis is used to control the cotangent complex $L_{(R/pR)/\Fp}$. For our non-$F$-finite $V_p$, $L_{V_p/\Fp}\simeq \Omega^1_{V_p/\Fp}[0]$ is flat but not finite projective. Indeed, the conormal
sequence for $V_p\twoheadrightarrow k=V_p/\mathfrak m$ gives
\[
    \mathfrak m/\mathfrak m^2
    \longrightarrow
    \Omega^1_{V_p/\Fp}\otimes_{V_p} k
    \longrightarrow
    \Omega^1_{k/\Fp}
    \longrightarrow 0.
\]
Since $k=\overline{\Fp}$, the last term vanishes, so the closed fiber has dimension at most $1$, which already rules out finite projectivity by comparison with the generic fiber. One can also use the discussion preceding \cite[Corollary~4.20]{BM_2023} to see this.
\end{remark}

\begin{remark}\label[remark]{rem:relation-B23}
\cite[Theorem~C]{B23} develops a general framework for $p$-Cartier smooth algebras which, specialized to the prism $(\Z_p,(p))$ and the algebra $V_p$, also yields the formula $H^n(\overline{\Prism}_{V_p/\Z_p})\simeq\Omega^n_{V_p/\Fp}\{-n\}$ of \cref{prop:Hn}.
Our approach uses \cite{BL22a} and Popescu's theorem, which is more straightforward in the situations we care about. The other results of this note: the counterexample itself, the failure of the limit argument, and the excellent repair, are independent of \cite{B23}.
\end{remark}

\bibliographystyle{alpha}
\bibliography{preprint}{}

\end{document}